\documentclass[reqno, 11pt]{amsart}
\usepackage[margin=1.1in]{geometry}
\usepackage{amsmath,amsthm,amssymb,enumerate,fancyhdr,mathtools}
\usepackage{enumitem}
\usepackage{physics}
\usepackage{amsmath,amsthm,amssymb,enumerate,mathtools}
\usepackage[numbers, comma, sort&compress]{natbib}
\usepackage{bbm}
\usepackage{enumitem}
\usepackage{comment}
\usepackage[toc,page]{appendix}
\usepackage{csquotes}
\usepackage{fullpage}
\usepackage[pdfpagemode=UseNone,bookmarksopen=false,colorlinks=true]{hyperref}
\hypersetup{colorlinks=true,linkcolor=blue,citecolor=blue,filecolor=blue,urlcolor=blue}
\usepackage{xcolor}

\newcommand{\weak}{\stackrel{w}{\longrightarrow}}

\newcommand{\one}{{\bf 1}}

\newcommand{\Gg}{\mathcal{G}}

\newcommand{\GG}{\mathbb{G}}
\newcommand{\Pro}{\mathbb{P}}

\newcommand{\bbe}{\protect{\mathbb E}}

\newcommand{\Var}{\textbf{Var}}

\newcommand{\bigo}{\mathrm{O}}
\newcommand{\Exp}{\mathbb{E}}
\newcommand{\smallo}{{\protect{\mathrm{o}}}}

\newcommand{\lep}{\left(}
\newcommand{\rip}{\right)}
\newcommand{\ERe}{Erd\H{o}s-R\'enyi }
\newcommand{\CLe}{Chung-Lu }
\newcommand{\ee}{\mathrm{e}}

\numberwithin{equation}{section}

\newtheorem{theorem}{Theorem}[section]
\newtheorem{lemma}[theorem]{Lemma}

\newtheorem{corollary}[theorem]{Corollary}
\newtheorem{remark}[theorem]{Remark}
\newtheorem{definition}[theorem]{Definition}

\newtheorem{assumption}[theorem]{Assumption}

\def\Var{{\rm Var}}
\def\E{{\rm E}}

\newcommand{\la}{\left\langle}
\newcommand{\ra}{\right\rangle}
\newcommand{\R}{\mathbb{R}}

\makeatletter
\let\oldabs\abs
\def\abs{\@ifstar{\oldabs}{\oldabs*}}

\numberwithin{equation}{section}

\begin{document}

\title[Spectrum of Chung-Lu random graphs]{Central limit theorem for the principal eigenvalue\\
and eigenvector of Chung-Lu random graphs}

\author[P. Dionigi]{Pierfrancesco Dionigi}
\address{Mathematical Institute, Leiden University, P.O.\ Box 9512, 2300 RA Leiden, The Netherlands.}
\email{p.dionigi@math.leidenuniv.nl}
\author[D. Garlaschelli]{Diego Garlaschelli}
\address{Lorentz Institute for Theoretical Physics, Leiden University, P.O.\  Box 9504, 2300 RA Leiden, The Netherlands
\& IMT School for Advanced Studies, Piazza S.\ Francesco 19, 55100 Lucca, Italy}
\email{garlaschelli@lorentz.leidenuniv.nl}
\author[R.~S.~Hazra]{Rajat Subhra Hazra}
\author[F.\ den Hollander]{Frank den Hollander}
\address{Mathematical Institute, Leiden University, P.O.\ Box 9512, 2300 RA Leiden, The Netherlands}
\email{r.s.hazra@math.leidenuniv.nl}
\email{denholla@math.leidenuniv.nl}

\author[M. Mandjes]{Michel Mandjes} 
\address{Korteweg-de Vries Institute, University of Amsterdam, P.O.\ Box 94248, 1090 GE Amsterdam, The Netherlands}
\email{M.R.H.Mandjes@uva.nl}
\thanks{PD, RSH, FdH and MM are supported by the Netherlands Organisation for Scientific Research (NWO) through the Gravitation-grant NETWORKS-024.002.003. DG is supported by the Dutch Econophysics Foundation (Stichting Econophysics, Leiden, The Netherlands) and by the European Union - Horizon 2020 Program under the scheme ``INFRAIA-01-2018-2019 - Integrating Activities for Advanced Communities'', Grant Agreement n.871042, ``SoBigData++: European Integrated Infrastructure for Social Mining and Big Data Analytics''.}
%
%
%
%
%

\begin{abstract}
A \CLe random graph is an inhomogeneous \ERe random graph in which vertices are assigned average degrees, and pairs of vertices are connected by an edge with a probability that is proportional to the product of their average degrees, independently for different edges. We derive a central limit theorem for the principal eigenvalue and the components of the principal eigenvector of the adjacency matrix of a \CLe random graph. Our derivation requires certain assumptions on the average degrees that guarantee connectivity, sparsity and bounded inhomogeneity of the graph. 
\end{abstract}

\keywords{Chung-Lu random graph; adjacency matrix; principal eigenvalue and eigenvector; central limit theorem}
\subjclass[2000]{60B20,60C05, 60K35 }

\date{\today}


\maketitle
%
%
%
%
%



\section{Introduction, main results and discussion}


\subsection{Introduction}
\label{sec:intro}


The spectral properties of adjacency matrices play an important role in various areas of network science. In the present paper we consider an inhomogeneous version of the \ERe random graph called the \emph{\CLe random graph} and we derive a central limit theorem for the principal eigenvalue and eigenvector of its adjacency matrix.


\subsubsection{Setting}

Recall that the homogeneous \ERe random graph has vertex set $[n] = \{1,\ldots,n\}$, and each edge is present with probability $p$ and absent with probability $1-p$, independently for different edges, where $p \in (0,1)$ may depend on $n$ (in what follows we often suppress the dependence on $n$ from the notation; the reader is however warned that most quantities depend on $n$). The average degree is the same for every vertex and equals $(n-1)p$ when self-loops are not allowed, and $np$ when self-loops are allowed (and are considered to contribute to the degrees of the vertices). In \cite{chungConnectedComponentsRandom2002} the following generalisation of the \ERe random graph is considered, called the \CLe random graph, with the goal to accommodate general degrees. Given a sequence of degrees $\vec{d}_n=(d_i)_{i \in [n]}$, consider the random graph $\Gg_n(\vec{d}_n)$ in which to each pair $i,j$ of vertices an edge is assigned independently with probability $p_{ij}=d_id_j/m_1$, where $m_1=\sum_{i=1}^n d_i$ (for computational simplicity we allow self-loops). Here, the degrees can act as vertex weights. Vertices with low weights are more likely to have less neighbours than vertices with high weights which act as hubs (see \cite[Chapter 6]{vanderhofstadRandomGraphsComplex2017} for a general introduction to generalised random graphs). If $m_\infty^2 \leq m_1$ with $m_\infty = \max_{i\in [n]} d_i$, then $p_{ij} \leq 1$ for all $i,j \in [n]$, and the sequence $\vec{d}_n$ is \emph{graphical}. Note that in $\Gg_n(\vec{d}_n)$ the \emph{expected} degree of vertex $i$ is $d_i$. The classical \ERe random graph (with self-loops) corresponds to $d_i=np$ for all $i \in [n]$.  


\subsubsection{Principal eigenvalue and eigenvector} 

The largest eigenvalue of the adjacency matrix $A$ and its corresponding eigenvector, written as $(\lambda_1,v_1)$, contain important information about the random graph. Several community detection techniques depend on a proper understanding of these quantities \cite{tulino2004random}, \cite{le2018concentration}, \cite{abbe2017community}, which in turn play an important role for various measures of network centrality \cite{martin2014localization}, \cite{newman2006finding}  and for the properties of dynamical processes (such as the spread of an epidemic) taking place on networks \cite{castellano2017,pastorsatorras2018}. For \ERe random graphs, it was shown in \cite{krivelevichLargestEigenvalueSparse2001} that \emph{with high probability} (whp) $\lambda_1$ scales like
\begin{equation}
\label{KS-result}
\lambda_1 \sim \max\{ \sqrt{D_\infty}, np\}, \qquad n \to \infty,
\end{equation}
where $D_\infty$ is the maximum degree. This result was partially extended to $\Gg_n(\vec{d}_n)$ in \cite{chungEigenvaluesRandomPower2003}, and more recently to a class of inhomogeneous \ERe random graphs in \cite{benaych-georgesLargestEigenvaluesSparse2017}, \cite{benaych-georgesSpectralRadiiSparse2021}.  For a related discussion on the behaviour of $(\lambda_1,v_1)$ in real-world networks, see \cite{castellano2017,pastorsatorras2018}. In the present paper we analyse the fluctuations of $(\lambda_1,v_1)$. We will be interested specifically in the case where $\lambda_1$ is detached from the bulk, which for \ERe random graphs occurs when $\lambda_1 \sim np$ whp, and for \CLe random graphs when $\lambda_1 \sim m_2/m_1$, where $m_2=\sum_{i\in [n]} d_i^2$. Note that the quotient $m_2/m_1$ arises from the fact that the average adjacency matrix is rank one and that its only non-zero eigenvalue is $m_2/m_1$. Such rank-one perturbations of a symmetric matrix with independent entries became prominent after the work in \cite{baikPhaseTransitionLargest2005}. Later studies extended this work to finite-rank perturbations \cite{baiCentralLimitTheorems2008}, \cite{benaych-georgesFluctuationsExtremeEigenvalues2011}, \cite{capitaineCentralLimitTheorems2012}, \cite{capitaineLargestEigenvaluesFinite2009}, \cite{feralLargestEigenvalueRank2007}, \cite{feralLargestEigenvaluesSample2009}. \ERe random graphs differ, in the sense that perturbations live on a scale different from $\sqrt{n}$. For \CLe random graphs we assume that $m_2/m_1\to\infty$. 

In the setting of inhomogeneous \ERe random graphs, finite-rank perturbations were studied in \cite{chakrabartyEigenvaluesOutsideBulk2020}. In that paper the connection probability between between $i$ and $j$ is given by $p_{ij} = \varepsilon_n f(i/n, j/n)$, where $f\colon\,[0,1]^2\to [0,1]$ is almost everywhere continuous and of finite rank, $\varepsilon_n \in [0,1]$ and $n\varepsilon_n\gg (\log n)^{8}$. However, for a \CLe random graph with a given degree sequence it is not always possible to construct an almost everywhere continuous function $f$ independent of $n$ such that $\varepsilon_n f(i/n, j/n) = d_i d_j/m_1$. In the present paper we extend the analysis in \cite{chakrabartyEigenvaluesOutsideBulk2020} to \CLe random graphs by focussing on $(\lambda_1,v_1)$. For \ERe random graphs it was shown in \cite{erdosSpectralStatisticsErdos2013}, \cite{erdosSpectralStatisticsErdosRenyi2012} that $\lambda_1$ satisfies a central limit theorem (CLT) and that $v_1$ aligns with the unit vector. These papers extend the seminal work carried out in \cite{furediEigenvaluesRandomSymmetric1981}.


\subsubsection{\CLe random graphs} 

In the present paper, subject to mild assumptions on $\vec{d}_n$, we extend the CLT for $\lambda_1$ from \ERe random graphs to \CLe random graphs, and derive a pointwise CLT for $v_1$ as well. It was shown in \cite{chungEigenvaluesRandomPower2003} that if $m_2/m_1\gg \sqrt{m_\infty}\,(\log n)$, then $\lambda_1 \sim m_2/m_1$ whp, while if  $\sqrt{m_\infty} \gg (m_2/m_1)( \log n)^2$, then $\lambda_1 = m_\infty$ whp. In fact, examples show that a result similar to \eqref{KS-result} does not hold, and that $\lambda_1$ does not scale like $\max\{m_2/m_1, \sqrt{m_{\infty}}\}$. These facts clearly show that the behaviour of $\lambda_1$ is controlled by subtle assumptions on the degree sequence. In what follows we stick to a bounded inhomogeneity regime where $m_2/m_1\asymp m_\infty$.

The behaviour of $v_1$ is interesting and challenging, and is of major interest for applications. One of the crucial properties to look for in eigenvectors is the phenomenon of localization versus delocalization. An eigenvector is called localized when its mass concentrates on a small number of vertices, and delocalized when its mass is approximately uniformly distributed on the vertices. The complete delocalization picture for \ERe random graphs was given in \cite{erdosSpectralStatisticsErdos2013}. In fact, it was proved that $\lambda_1$ is close to the scaled unit vector in the $\ell_\infty$-norm. In the present paper we do not study localization versus delocalization for \CLe random graphs in detail, but we do show that in a certain regime there is strong evidence for delocalization because $v_1$ is close to the scaled unit vector. In \cite[Corollary 1.3 ]{bourgade2017eigenvector} the authors found that the eigenvectors of a Wigner matrix with independent standard Gaussian entries are distributed according to a Haar measure on the orthogonal group, and the coordinates have Gaussian fluctuations after appropriate scaling. Our work shows that the coordinate-wise fluctuations hold as well for the principal eigenvector of the non-centered \CLe adjacency matrix and that they are Gaussian after appropriate centering and scaling.

\subsubsection{Outline}

In Section \ref{sub:model} we define the \CLe random graph, state our assumption on the degree sequence, and formulate two main theorems: a CLT for the largest eigenvalue and a CLT for its associated eigenvector. In Section \ref{sec:discussion} we discuss these theorems and place them in their proper context. Section \ref{sec:proofs-1} contains the proof of the CLT of the eigenvalue and Section \ref{sub:cltvect} studies the properties of the principal eigenvector.


\subsection{Main results}
\label{sub:model}


\subsubsection{Set-up}

Let $\GG_n$ be the set of simple graphs with $n$ vertices. Let $\vec{d}_n=(d_i)_{i \in [n]}$ be a sequence of degrees, such that $d_i\in \mathbb{N}$ for all $i\in [n]$ and abbreviate 
\[
m_k = \sum_{i \in [n]} (d_i)^k, \qquad m_\infty = \max_{i \in [n]} d_i, \qquad  m_0= \min_{i\in [n]} d_i,
\]
Note that these numbers depend on $n$, but in the sequel we will suppress this dependence. For each pair of vertices $i,j$ (not necessarily distinct), we add an edge independently with probability
\begin{equation}
\label{eq:chunglu}
p_{ij}=\frac{d_id_j}{m_1}.
\end{equation}
The resulting random graph, which we denote by $\mathcal{G}_n(\vec{d}_n)$, is referred to in the literature as the Chung-Lu random graph. In \cite{chungConnectedComponentsRandom2002} it was assumed that $m_\infty^2 \le m_1$ to ensure that $p_{ij} \leq 1$. In the present paper we need sharper restrictions. 

\begin{assumption}
\label{ass:deg}
{\rm Throughout the paper we need two assumptions on $\vec{d}_n$ as $n\to\infty$:
\begin{enumerate}
\item[(D1)]
{\bf Connectivity and sparsity:} There exists a $\xi>2$ such that 
\[
(\log n)^{2\xi} \ll m_\infty \ll n^{1/2}.
\]
\item[(D2)]
{\bf Bounded inhomogeneity:} $m_0 \asymp m_\infty$.
\end{enumerate}
}
\hfill$\spadesuit$
\end{assumption}

\noindent
The lower bound in Assumption \ref{ass:deg}(D1) guarantees that the random graph is connected whp and that it is not too sparse. The upper bound is needed in order to have $m_\infty = \smallo(\sqrt{m_1})$, which implies that \eqref{eq:chunglu} is well defined. Assumption \ref{ass:deg}(D2) is a restriction on the inhomogeneity of the model and requires that the smallest and the largest degree are comparable. 

\begin{remark}
{\rm The lower bound on $m_\infty$ in Assumption \ref{ass:deg}(D1) can be seen as an adaptation to our setting of the main condition in \cite[Theorem 2.1]{chungEigenvaluesRandomPower2003} for the asymptotics of $\lambda_1$. As mentioned in Section~\ref{sec:intro}, under the assumption
\[
\frac{m_2}{m_1}\gg \sqrt{m_\infty}\,(\log n)^\xi, 
\]
 \cite{chungEigenvaluesRandomPower2003} shows that $\lambda_1= [1+\smallo(1)]\,m_2/m_1$ whp. It is easy to see that the above condition together with Assumption \ref{ass:deg}(D2) gives the lower bound in Assumption \ref{ass:deg}(D1).} \hfill$\spadesuit$
\end{remark} 

\begin{remark}
{\rm When $m_\infty\ll n^{1/6}$, \cite[Theorem 6.19]{vanderhofstadRandomGraphsComplex2017} implies that our results also hold for the \emph{Generalized Random Graph} (GRG) model with the same average degrees. This model is defined by choosing connection probabilities of the form
\[
p_{ij}=\frac{d_id_j}{m_1+d_id_j},
\] 
and arises in statistical physics as the \emph{canonical ensemble} constrained on the expected degrees, which is also called the \emph{canonical configuration model}. Note that in the above connection probability, $d_i$ plays the role of a hidden variable, or a Lagrange multiplier controlling the expected degree of vertex $i$, but does not in general coincide with the expected degree itself. However, under the assumptions considered here, $d_i$ does coincide with the expected degree asymptotically. The reader can find more about GRG and their use in \cite[Chapter 6]{vanderhofstadRandomGraphsComplex2017}, and about their role in statistical physics in \cite{squartiniAnalyticalMaximumlikelihoodMethod2011}. In the corresponding \emph{microcanonical ensemble} the degrees are not only fixed in their expectation but they take a precise deterministic value, which corresponds to the \emph{microcanonical configuration model}. The two ensembles were found to be \emph{nonequivalent} in the limit as $n\to\infty$ \cite{squartiniBreakingEnsembleEquivalence2015}. This result  was shown to imply a finite difference between the expected values of the largest eigenvalue $\lambda_1$ in the two models \cite{dionigiSpectralSignatureBreaking2021} when the degree sequence was chosen to be constant ($d_i=d$ for all $i \in [n]$). In this latter case the canonical ensemble reduces to the \ERe random graph with $p=d/n$, while the microcanonical ensemble reduces to the $d$-regular random graph model. Although ensemble nonequivalence is not our main focus here, we will briefly relate some of our results to this phenomenon.} \hfill$\spadesuit$
\end{remark} 


\subsubsection{Notation}

Let $A$ be the adjacency matrix of $\Gg_n(\vec{d}_n)$ and $\Exp[A]$ its expectation. The $(i,j)$-th entry of $\Exp[A]$ equals to $p_{ij}$ in \eqref{eq:chunglu}. The $(i,j)$-th entry of $A-\Exp[A]$ is an independent centered Bernoulli random variable with parameter $p_{ij}$. Let $\lambda_1\ge \ldots\ge \lambda_n$ be the eigenvalues of $A$ and let $v_1, \ldots, v_n$ be the corresponding eigenvectors. The vector $e$ will the $n\times 1$ row vector given by
\begin{equation}
\label{eq:def-e}
e=\frac{1}{\sqrt{m_1}}( d_1, \cdots, d_n)^t,
\end{equation}
where $t$ stands for transpose. It is easy to see that $\Exp[A]=ee^t$. 
\begin{definition}
{\rm Following \cite{erdosSpectralStatisticsErdos2013}, we say that an event $\mathcal{E}$ holds \emph{with high probability} (whp) when there exist $\xi>2$ and $\nu>0$ such that}
\begin{equation}
\label{whp}
\Pro(\mathcal{E}) \leq \ee^{-\nu (\log n)^\xi}.
\end{equation}
\end{definition}

\noindent
Note that this is different from the classical notion of whp, because it comes with a specific rate. We write $\weak$ to denote weak convergence as $n\to\infty$, and use the symbols $\smallo,\bigo$ to denote asymptotic order for sequences of real numbers. 


\subsubsection{CLT for the principal eigenvalue}

Our first theorem identifies two terms in the expectation of the largest eigenvalue, and shows that the largest eigenvalue follows a central limit theorem.

\begin{theorem}
\label{thm:CLTlambda}
Under Assumption \ref{ass:deg}, the following hold:
\begin{itemize}
\item[{\rm (I)}]
\[
\Exp[\lambda_1] = \frac{m_2}{m_1} + \frac{m_1m_3}{m_2^2} +\smallo(1), \qquad n\to\infty.
\]
\item[{\rm (II)}]
\[
\frac{m_2}{m_1}\left(\frac{\lambda_1-\Exp[\lambda_1]}{\sigma_1}\right) \weak \mathcal{N}(0,2), \qquad n \to \infty,
\]
where 
\begin{align*}
\sigma_1^2 = \sum_{i,j} (p_{ij})^3(1-p_{ij}) \sim \frac{m_3^2}{m_1^3},
\qquad n \to \infty.
\end{align*}
\end{itemize}
\end{theorem}


\subsubsection{CLT for the principal eigenvector}

Our second theorem shows that the principal eigenvector is parallel to the normalised degree vector, and is close to this vector in $\ell^\infty$-norm. It also identifies the expected value of the components of the principal eigenvector, and shows that the components follow a central limit theorem.

\begin{theorem}
\label{thm:CLTv}
Let $\tilde{e} = e\sqrt{m_1/m_2}$ be the $\ell^2$-nomalized degree vector. Let $v_1$ be the eigenvector corresponding to $\lambda_1$ and let $v_1(i)$ denote the $i$-th coordinate of $v_1$.  Under Assumption \ref{ass:deg}, the following hold:
\begin{itemize}
\item[{\rm (I)}]
$\langle v_1,\tilde{e} \rangle = 1+\smallo(1)$ as $n\to\infty$ whp.
\vskip10pt
\item[{\rm (II)}]
$\|v_1- \tilde{e}\|_{\infty} \leq \bigo\lep\frac{(\log n)^{\xi}}{\sqrt{n m_{\infty}}}\rip$ as $n\to\infty$ whp.
\vskip10pt
\item[{\rm (III)}]
$\Exp[v_1(i)]=\frac{d_i}{\sqrt{m_2}}+\bigo\lep\frac{(\log n)^{2\xi}}{\sqrt{m_2}}\rip$ as $n\to\infty$. 
\end{itemize}
Moreover, if the lower bound in Assumption \ref{ass:deg}(D1) is strengthened to $(\log n)^{4\xi} \ll m_\infty$, then for all $i \in [n]$,
\begin{itemize}
\item[{\rm(IV)}]
\[
\frac{m_2}{m_1}\left(\frac{v_1(i)-d_i/\sqrt{m_2}}{s_1(i)}\right) 
\weak \mathcal{N}(0,1), \qquad n \to \infty, 
\]
where \[s_1^2(i)=\sum_j d_j^2 p_{ij}(1-p_{ij}) \sim d_i\frac{m_3}{m_1},\qquad n \to \infty.\]
\end{itemize}
\end{theorem}


\subsection{Discussion}
\label{sec:discussion}

We place the theorems in their proper context.

\medskip\noindent
{\bf 1.}
Theorems~\ref{thm:CLTlambda}--\ref{thm:CLTv} provide a CLT for $\lambda_1,v_1$. We note that $m_2/m_1$ is the leading order term in the expansion of $\lambda_1$, while $m_1m_3/m_2^2$ is a correction term. We observe that Theorem~\ref{thm:CLTlambda}(I) does not follow from the results in \cite{chungEigenvaluesRandomPower2003}, because the largest eigenvalue need not be uniformly integrable and also the second order expansion is not considered there. We also note that in Theorem~\ref{thm:CLTlambda}(II) the centering of the largest eigenvalue, $\Exp[\lambda_1]$, cannot be replaced by its asymptotic value as the error term is not compatible with the required variance. 

\medskip\noindent
{\bf 2.}
The lower bound in Assumption \ref{ass:deg}(D1) is needed to ensure that the random graph is connected, and is crucial because the largest eigenvalue is very sensitive to connectivity properties. Assumption \ref{ass:deg}(D2) is needed to control the inhomogeneity of the random graph. It plays a crucial role in deriving concentration bounds  on the central moments $\langle e, (A-\Exp[A])^k e\rangle$, $k \in \mathbb{N}$, with the help of a result from \cite{erdosSpectralStatisticsErdos2013}. Further refinements may come from different tools, such as the non-backtracking matrices used in \cite{benaych-georgesLargestEigenvaluesSparse2017}, \cite{benaych-georgesSpectralRadiiSparse2021}. While Assumption \ref{ass:deg}(D1) appears to be close to optimal, Assumption \ref{ass:deg}(D2) is far from optimal. It would be interesting to allow for empirical degree distributions that converge to a limiting degree distribrution with a power law tail.  

\medskip\noindent
{\bf 3.}
 As already noted, if the expected degrees are all equal to each other, i.e., $d_i=d$ for all $i \in [n]$, then the \CLe random graph, or canonical configuration model, reduces to the homogeneous \ERe random graph with $p=d/n$, while the corresponding microcanonical configuration model reduces to the homogeneous $d$-regular random graph model (here, all models allow for self-loops). This implies that, for the homogeneous \ERe random graph with connection probability $p \gg (\log n)^{2\xi}/n$, $\xi>2$, Theorem~\ref{thm:CLTlambda}(I) reduces to
\[
\mathbb{E} [\lambda_1] = np+1+\smallo(1), \qquad n\to\infty,
\]
while Theorem \ref{thm:CLTlambda}(II) reduces to 
\[
\frac{1}{\sqrt{p}} \left(\lambda_1- \Exp[\lambda_1]\right) \weak \mathcal{N}(0, 2), \qquad n \to \infty.
\]
Both these properties were derived in \cite{erdosSpectralStatisticsErdosRenyi2012} for homogeneous \ERe random graphs and also for rank-1 perturbations of Wigner matrices. In \cite{dionigiSpectralSignatureBreaking2021}, the fact that $\Exp[\lambda_1]$ in the canonical ensemble differs by a finite amount from the corresponding expected value (here, $d=np$) in the microcanonical ensemble ($d$-regular random graph) was shown to be a signature of ensemble nonequivalence.

\medskip\noindent
{\bf 4.}
In case $d_i=d$ for all $i \in [n]$, Theorem \ref{thm:CLTv}(III) reduces to the following CLT, which was not covered by \cite{erdosSpectralStatisticsErdosRenyi2012} and \cite{dionigiSpectralSignatureBreaking2021}.

\begin{corollary}
\label{CLTeigvec_erdos}
For the \ERe random graph with $(\log n)^{4\xi}/n \ll p \ll n^{-1/2}$ for some $\xi>2$,
\[
n\sqrt{\frac{p}{1-p}}\lep v_1(i)-\frac{1}{\sqrt{n}}\rip \weak \mathcal{N}(0,1), \qquad n \to \infty.
\]
\end{corollary}
\noindent Note that, in the corresponding microcanonical ensemble ($d$-regular random graph), $v_1$ coincides with the constant vector where $v_1(i)=1/\sqrt{n}$ for all $i \in [n]$. Therefore in the canonical ensemble each coordinate $v_1(i)$ has Gaussian fluctuations around the corresponding deterministic value for the microcanonical ensemble. This behaviour is similar to the degrees having, in the canonical configuration model, either Gaussian (in the dense setting) or Poisson (in the sparse setting) fluctuations around the corresponding deterministic degrees for the microcanonical configuration model \cite{garlaschelliCovarianceStructureBreaking2018}.

\medskip\noindent
{\bf 5.}
In \cite{CHdHS2021} the empirical spectral distribution of $A$ was considered under the assumption that 
\[
(m_\infty)^2/m_1 \ll 1 \ll  n(m_\infty)^2/m_1,
\] 
which is weaker than Assumption \ref{ass:deg}. It was shown that if $\mu_n \weak \mu$ with $\mu_n = n^{-1} \sum_{i=1}^n \delta_{d_i/m_\infty}$ and $\mu$ some probability distribution on $\R$, then 
\[
\mathrm{ESD}\left(\frac{A}{\sqrt{n(m_\infty)^2/m_1}}\right) \weak \mu \boxtimes \mu_\mathrm{sc}
\]  
with $\mu_{\mathrm{sc}}$ the Wigner semicircle law and $\boxtimes$ the free multiplicative convolution. Since $\mu\boxtimes \mu_{\mathrm{sc}}$ is compactly supported, this shows that the scaling for the largest eigenvalue and the spectral distribution are different.


\section{Proof of Theorem \ref{thm:CLTlambda}}
\label{sec:proofs-1}

In what follows we use the well-known method of writing the largest eigenvalue of a matrix as a rank-1 perturbation of the centered matrix. This method was previously successfully employed in \cite{furediEigenvaluesRandomSymmetric1981, pecheLargestEigenvalueSmall2006, erdosSpectralStatisticsErdos2013}.

Given the adjacency matrix $A$ of our graph $G$, we can write $A=H+\Exp[A]$ with $H=A-\Exp[A]$. Let $v_1$ be the eigenvector associated with the eigenvalue $\lambda_1$. Then
\[
A v_1 =\lambda_1 v_1, \quad
(H+\Exp[ A])v_1 =\lambda_1v_1, \quad
(\lambda_1 I - H) v_1 = \Exp[A] v_1.
\]
Using that $\Exp[A]=ee^t$, we have $(\lambda_1 I - H) v_1=  \la e,  v_1\ra  e$, where $I$ is the $n \times n$ identity matrix. It follows that if $\lambda_1$ is not an eigenvalue of $H$, then the matrix $(\lambda_1 I - H)$ is invertible, and so  
\begin{equation}\label{eq:eigenvector}
v_1=  \la e,  v_1\ra (\lambda_1 I - H)^{-1} e.
\end{equation}
Eliminating the eigenvector $v_1$ from the above equation, we get
\[
1 =  \la e,   (\lambda_1 I - H)^{-1} e\ra ,
\]
where we use that $\la e,  v_1\ra  \neq 0$ (since $\lambda_1$ is not an eigenvalue of $H$). Note that this can be expressed as
\begin{equation} 
\label{eigenvalue2}
\lambda_1 =  \la e,   \left(I- \frac{H}{\lambda_1}\right)^{-1} e\ra  
= \sum_{k=0}^{\infty}  \left\langle e, \left(\frac{H}{\lambda_1}\right)^{k}e\right\rangle \quad whp,
\end{equation}
where the validity of the series expansion will be an immediate consequence of Lemma \ref{lem:concentrationlambda} below.

Section \ref{sub:spectral_H} derives bounds on the spectral norm of $H$. Section~\ref{sub:EA} analyses the expansion in \eqref{eigenvalue2} and prove the scaling of $\Exp[\lambda_1]$. Section \ref{sub:cltlambda} is devoted to the proof of the CLT for $\lambda_1$, Section \ref{sub:cltvect} to the proof of the CLT for $v_1$. In the expansion we distinguish three ranges: (i) $k=0,1,2$; (ii) $3 \leq k \leq L$; (iii) $L < k < \infty$, where 
\[
L=\lfloor \log n\rfloor.
\] 
We will show that (i) controls the mean and the variance in both CLTs, while (ii)-(iii) are negligible error terms.


\subsection{The spectral norm}
\label{sub:spectral_H}

In order to study $\lambda_1$, we need good bounds on the spectral norm of $H$. The spectral norm of matrices with inhomogeneous entries has been studied in a series of papers \cite{benaych-georgesLargestEigenvaluesSparse2017}, \cite{benaych-georgesSpectralRadiiSparse2021}, \cite{altExtremalEigenvaluesCritical2020} for different density regimes.

An important role is played by $\lambda_1(\Exp[A])$. In recent literature this quantity has been shown to play a prominent role in the so-called BBP-transition \cite{baikPhaseTransitionLargest2005}. Given our setting \eqref{eq:chunglu}, it is easy to see that 
\begin{equation}
\label{eq:largestexmat}
\lambda_1(\Exp[A])=\frac{m_2}{m_1},
\end{equation}
while all other eigenvalues of $\Exp[A]$ are zero. 

\begin{remark}
\label{rmk:m2/m1}
Since $m_0\leq \frac{m_2}{m_1}\leq m_\infty$, {\rm Assumption \ref{ass:deg}(D2) implies that 
\begin{equation}
\label{eq:m2/m1asymp}
\frac{m_2}{m_1}\asymp m_\infty.
\end{equation}
} \hfill$\spadesuit$
\end{remark}

We start with the following lemma, which ensures concentration of $\lambda_1$ and is a direct consequence of the results in \cite{benaych-georgesSpectralRadiiSparse2021} (which match Assumption \ref{ass:deg}).

\begin{lemma}
\label{lem:concentrationlambda}
Under Assumption \ref{ass:deg}, whp
\[
\left|\frac{\lambda_1(A)-\lambda_1(\Exp[A])}{\lambda_1(\Exp[A])}\right|
 = \bigo\lep\frac{1}{\sqrt{m_\infty}}\rip, \qquad n \to \infty,
\]
and consequently
\[
\frac{\lambda_1(A)}{\lambda_1\lep\Exp[A]\rip} \overset{\Pro}\to 1, \qquad n \to \infty.
\]
 \end{lemma}

\begin{proof}
In the proof it is understood that all statements hold whp in the sense of \eqref{whp}. Let $A=H+\Exp[A]$. Due to Weyl's inequality, we have that 
 \[
 \lambda_1(\Exp[A])-\|H\|\leq\lambda_1(A)\leq\lambda_1(\Exp[A])+\|H\|.
 \]
 From \cite[Theorem 3.2]{benaych-georgesSpectralRadiiSparse2021} we know that there is a universal constant $C>0$ such that
 \[
 \Exp\left[\|H\|\right]\leq \sqrt{m_\infty}\lep2+\frac{C}{q}\sqrt{\frac{\log n}{1\vee\log\lep\frac{\sqrt{\log n}}{q}\rip}}\rip,
 \]
 where 
 \[
 q=\sqrt{m_\infty}\wedge n^{1/10} \kappa^{-1/9}
 \]
 with $\kappa$ defined by 
 \[
\kappa = \max_{ij} \frac{p_ij}{m_\infty/n}=\frac{n m_\infty}{m_1}.
 \] 
Thanks to Assumption \ref{ass:deg}(D2), we have $\kappa=\bigo(1)$. By Remark 3.1 of \cite[Remark 3.1]{benaych-georgesSpectralRadiiSparse2021} (which gives us that $q=\sqrt{m_\infty}$ for $n$ large enough) and Assumption \ref{ass:deg}, we get that
\begin{equation}
\label{eq:spectralnorm}
\Exp\left[\|H\|\right] \leq 
\begin{cases} 
\sqrt{m_\infty}\lep2+\frac{C\sqrt{\log n}}{\sqrt{m_\infty}}\rip, \qquad (\log n)^{2\xi}\leq\sqrt{m_\infty}
\leq n^{1/10} \kappa^{-1/9},\\[0.2cm]
\sqrt{m_\infty}\lep2+\frac{C'\sqrt{\log n}}{n^{1/10}}\rip, \qquad \sqrt{m_\infty}\geq n^{1/10} \kappa^{-1/9}.
\end{cases}
\end{equation}
Using \cite[Example 8.7]{boucheronConcentrationInequalitiesNonasymptotic2013} or \cite[Equation 2.4]{benaych-georgesSpectralRadiiSparse2021} (the Talagrand inequality), we know that there exists a universal constant $c>0$ such that
\[
\Pro\lep \left|\|H\|-\Exp[\|H\|]\right|>t\rip\leq 2\ee^{-ct^2}.
\]
For $t=\sqrt{\nu (\log n)^\xi}$,
\begin{equation}
\label{eq:normH}
\Exp[\|H\|]-\sqrt{\nu} (\log n)^{\xi/2}\leq\|H\|\leq \Exp[\|H\|]+\sqrt{\nu} (\log n)^{\xi/2}.
\end{equation}
Thus, we have
\begin{equation}
\label{eq:normH2}
\left|\lambda_1(A)-\lambda_1(\Exp[A])\right|\le \|H\|\le \sqrt{m_{\infty}}(2+\smallo(1))+ \sqrt{\nu} (\log n)^{\xi/2}. 
\end{equation}
Using that $\lambda_1(\Exp[A])= m_2/m_1$, we have that whp the following bound holds:
\[
\left|\frac{\lambda_1(A)-\lambda_1(\Exp[A])}{\lambda_1(\E[A])}\right|
\le \frac{\sqrt{m_\infty}}{m_2/m_1}\lep 2+\smallo(1)\rip+\frac{ \sqrt{\nu} (\log n)^{\xi/2}}{m_2/m_1}
= \bigo\lep\frac{\sqrt{m_\infty}}{m_2/m_1}\rip.
\]
Via Assumption \ref{ass:deg} and \eqref{eq:m2/m1asymp} the claim follows.
\end{proof}

\begin{remark}
\label{rem:obs}
$\mbox{}$
{\rm 
\begin{itemize}
\item[(a)] 
The proof of Lemma \ref{lem:concentrationlambda} works well if we replace Assumption \ref{ass:deg}(D2) by a milder condition. Indeed, the former is directly linked to the parameter $\kappa$ that appears in the proof of Lemma \ref{lem:concentrationlambda} and in the proof of \cite[Theorem 3.2]{benaych-georgesSpectralRadiiSparse2021}, which contains a more general condition on the inhomogeneity of the degrees.
\item[(b)] 
Note that a consequence of proof of Lemma \ref{lem:concentrationlambda} is that whp
\begin{equation}
\label{eq:normbound}
\frac{\|H\|}{\lambda_1(A)}\le 1-C_0
\end{equation}
for some $C_0\in (0,1)$. This allows us to claim that whp the inverse
\begin{equation}
\label{eq:inverse}
\left(I- \frac{H}{\lambda_1(A)}\right)^{-1}
\end{equation}
exists.
\end{itemize}
}\hfill$\spadesuit$
\end{remark}

\begin{lemma}
\label{lemma:exppowers}
Let $1\le k\le L$. Then, under Assumption \ref{ass:deg}, whp 
\[
\left| \la e,   H^k e\ra  - \bbe \left[\la e,  H^k e\ra \right]\right| \leq  
C \frac{m_2}{m_1}\frac{ m_\infty^{\frac{k}{2}}(\log n)^{k\xi}}{\sqrt{n}},
\]
i.e.,
\[
\max_{1 \leq k \leq L} \Pro\left( \left|\la e,   H^k e\ra  - \bbe \left[\la e,  H^k e\ra \right]\right|
>  \frac{C (\log n)^{k\xi}m_\infty^{\frac{k}{2}}}{\sqrt{n}} \frac{m_2}{m_1}\right) \leq \ee^{-\nu(\log n)^\xi},
\qquad n\geq n_1(\nu,\xi).
\]
\end{lemma}

\noindent
Lemma \ref{lemma:exppowers} is a generalization to the inhomogeneous setting of \cite[Lemma 6.5]{erdosSpectralStatisticsErdos2013}. We skip the proof because it requires a straightforward modification of the arguments in \cite{erdosSpectralStatisticsErdos2013}.

\begin{lemma}
\label{lemma:expectationbounds}
Under Assumption \ref{ass:deg}, for $2\le k\le L$, there exists a constant $C>0$ such that
\begin{equation}
\label{eq:expectationbounds}
\Exp\left[ \la e,   H^k e\ra\right] \le \frac{m_2}{m_1} (Cm_\infty)^{k/2}.
\end{equation}
\end{lemma}

\begin{proof}
Let $\mathcal E$ be the high probability event defined by \eqref{eq:normH}, i.e.,
\[
\| H\| \le \E[\|H\|] + \sqrt{\nu}(\log n)^{\xi/2}\le m_\infty\left( 1+ \bigo\left( \frac{(\log n)^{\xi/2}}{m_\infty}\right)\right).
\]
Due to Assumption \ref{ass:deg}(D1) we can bound the right-hand side by $Cm_\infty$. Since $\|e\|_2^2= m_2/m_1$, on this event we have
\[
\Exp\left[ \left( \la e,   H^k e\ra\right)\one_{\mathcal E}\right]
\le \|e\|_2^2\, \Exp[\|H\|^k\one_{\mathcal E}]\\
\le \frac{m_2}{m_1} (Cm_\infty)^{k/2} .
\]
We show that the expectation when evaluated on the complementary event is negligible. Indeed, observe that
\begin{align*}
\Exp\left[  \la e,   H^k e\ra\right]
&= \Exp\left(\sum_{i_1, \ldots, i_{k+1}=1}^n e_{i_1}e_{i_{k+1}} \prod_{j=1}^k H(i_j, i_{j+1})\right)^2\\
&\leq \left( \frac{n^{k+1} m_\infty^2}{m_1}\right)^2\le C\ee^{(2k+2)\log n} \le \ee^{2(\log n)^2},
\end{align*}
where in the last inequality we use that $m_\infty=\smallo(\sqrt{m_1})$. This, combined with the exponential decay of the event $\mathcal E^c$, gives
\[
\Exp\left[ \la e,   H^k e\ra\one_{A^c}\right] \le C\ee^{- \nu(\log n)^{\xi}},
\]
and so the claim follows.
\end{proof}


\subsection{Expansion for the principal eigenvalue}
\label{sub:EA}

We denote the event in Lemma \ref{lem:concentrationlambda} by $\mathcal E$, which has high probability. As noted in Remark \ref{rem:obs}(b), $I-\frac{H}{\lambda_1}$ is invertible on $\mathcal E$. Hence, expanding on $\mathcal E$, we get
\[
\lambda_1= \sum_{k=0}^\infty \la e,  \frac{H^k}{\lambda_1^k} e\ra .
\]
We split the sum into two parts:
\begin{equation}
\label{eq:masterequation}
\lambda_1 = \sum_{k=0}^{L}  \frac{ \la e,   H^k e\ra}{\lambda_1^{k}} +  \sum_{k=L+1}^\infty \frac{\la e,   H^k e\ra}{\lambda_1^{k}}.
\end{equation}

First we show that we may ignore the second sum. To that end we observe that, by Assumption \ref{ass:deg} (D1), on the event $\mathcal E$ we can estimate
\begin{align}
\left| \sum_{k= L+1}^\infty \frac{ \la e,   H^k e\ra}{\lambda_1^{k}}\right|
&\le \sum_{k=L+1}^\infty \frac{\|e\|_2^2 \|H\|^k}{\lambda_1^k}
\le \sum_{k=L+1}^\infty \frac{m_2}{m_1}  \frac{m_{\infty}^{k/2}}{ (Cm_2/m_1)^k}\nonumber\\
&\le  \sum_{k=L+1}^\infty \frac{C'}{ m_{\infty}^{ k/2-1}}=\bigo\left(\ee^{-c\log \sqrt{n}}\right).
\label{eq:tailsum}
\end{align}
Because of \eqref{eq:tailsum} and the fact that $\Exp(\la e,H e\ra)=0$, \eqref{eq:masterequation} reduces to
\begin{align*}
\lambda_1
&=\sum_{k=3}^{L}  \frac{\Exp \left[\la e, H^k e\ra\right]}{\lambda_1^{k}}
+ \sum_{k=3}^{L}  \frac{ \la e, H^k e\ra - \Exp\left[\la e, H^k e\ra\right]}{\lambda_1^{k}}\\
&\qquad +\la e, e\ra
+ \frac{1}{\lambda_1} \la e, H e\ra + \frac{1}{\lambda_1^2} \la e, H^2 e\ra+\smallo(1).
\end{align*}

Next, we estimate the second sum in the above equation. Using Lemma \ref{lem:concentrationlambda}, we get
\begin{align*}
&\left|\sum_{k=3}^{L}  \frac{ \la e, H^k e\ra  - \Exp \left[\la e, H^k e\ra \right]}{\lambda_1^{k}}\right|\\
&\le \sum_{k=3}^{L}  \frac{C  m_{\infty}^{\frac{k}{2}}(\log n)^{k\xi}}{\sqrt{n}(m_2/m_1)^{k-1}}
\le \sum_{k=3}^{L}  \frac{C(\log n)^{k\xi}}{\sqrt{n}m_{\infty}^{k/2-1}}
\le\bigo\left( \frac{C(\log n)^{\xi+1}}{\sqrt{nm_\infty}}\right) = \smallo(1).
\end{align*} 
From Lemma \ref{lemma:expectationbounds} we have
\[
\sum_{k=3}^L \frac{\Exp \la e,   H^k e\ra }{\lambda_1^{k}}
\le \sum_{k=3}^L \frac{ \frac{m_2}{m_1} (Cm_{\infty})^{k/2} }{\left(m_2/m_1\right)^k}
=\bigo\left( \frac{1}{\sqrt{m_{\infty}}}\right)= \smallo(1),
\]
where the last estimate follows from Assumption \ref{ass:deg}(D1). Hence, on $\mathcal E$,
\[
\lambda_1= \la e, e\ra  +\frac{1}{\lambda_1} \la e,   H e\ra+ \frac{\la e,   H^2 e\ra}{\lambda_1^2}
+  \smallo(1).
\]

Iterating the expression for $\lambda_1$ in the right-hand side, we get
\begin{align*}
\lambda_1
&= \la e,   e\ra +  \la e,   H e\ra \left(  \la e, e\ra  + \frac{1}{\lambda_1}\la e,  He\ra +\frac{1}{\lambda_1^2} \la e,   H^2 e\ra
+ \smallo (1) \right)^{-1} \\
&\quad +  \la e,   H^2 e\ra \left(  \la e, e\ra + \frac{1}{\lambda_1}\la e,   H e\ra
+\frac{1}{\lambda_1^2} \la e,   H^2 e\ra + \smallo (1) \right)^{-2} + \smallo (1),
\end{align*}
Expanding the second and third term we get,
\begin{align*}
\lambda_1&=  \la e, e\ra  + \frac{\la e,   H e\ra}{\la e, e\ra }\left( 1 - \frac{ \la e, H e\ra }{\lambda_1 \la e, e\ra }
-\frac{\la e, H^2 e\ra }{\lambda_1^2 \la e, e\ra } + \smallo (1) \right) \\
&\quad + \frac{\la e,  H^2 e\ra }{(\la e, e\ra )^2}\left( 1 - \frac{2\la e, H e\ra }{\lambda_1 \la e, e\ra }
-\frac{2\la e, H^2 e\ra }{\lambda_1^2 \la e, e\ra} + \smallo (1) \right) + \smallo (1),\\
&=  \la e, e \ra + \frac{\la e, H e\ra}{ \la e, e\ra } - \frac{\la e,  H e\ra^2}{\lambda_1 \la e,  e\ra^2}
+ \frac{\la e, H^2 e\ra }{ \la e, e\ra^2} + \smallo (1).
\end{align*}
Here we use that $\la e, e\ra= m_2/m_1\to \infty$, and we ignore several other terms because they are small whp, for example,
$$
\frac{\la e, H e\ra \la e, H^2 e\ra}{\lambda_1^2\la e, e\ra^2} = \bigo\left( \frac{m_\infty^{3/2}}{(m_2/m_1)^4}\right)=\smallo(1).
$$
One more iteration gives
\begin{align*}
\lambda_1
&=  \la e, e \ra + \frac{\la e, H e\ra}{\la e, e \ra} + \frac{\la e, H^2 e\ra}{\la e, e \ra^2} \\
&\quad - \frac{\la e, H e\ra^2}{\la e, e \ra^2}\left(  \la e, e \ra + \frac{1}{\lambda_1}\la e, H e\ra
+\frac{1}{\lambda_1^2} \la e, H^2 e\ra + \smallo (1) \right)^{-1} + \smallo (1) \\
&=  \la e, e \ra + \frac{\la e, H e\ra}{\la e, e \ra} + \frac{\la e, H^2 e\ra}{ \la e, e \ra^2}
- \frac{\la e, He\ra^2}{\la e, e \ra^3} + \frac{\la e, H^2 e\ra^2 \la e, H e\ra}{\lambda_1 \la e, e \ra^3}
+ \frac{\la e, H^2 e\ra^3}{\lambda_1^2 \la e, e \ra^3} + \smallo (1).
\end{align*}

\begin{proof}[{\bf Proof of Theorem \ref{thm:CLTlambda} (I)}]
Since the probability of $\mathcal E^c$ decays exponentially with $n$, taking the expectation of the above term and using that $\Exp[e^\prime W e]=0$, we obtain
\begin{align*}
\mathbb{E} [\lambda_1] =  \la e, e \ra + \frac{\mathbb{E}[\la e, H^2 e\ra ]}{ \la e, e \ra^2}
- \frac{\mathbb{E}[\la e, He\ra ^2]}{ \la e, e \ra^3} + \smallo (1)
= \frac{m_2}{m_1}+ \frac{m_1m_3}{m_2^2}- \frac{ m_3^2}{m_2^3}
+ \smallo(1).
\end{align*}

Note that 
$$
\frac{ m_3^2}{m_2^2}\le \frac{m_{\infty}^2}{n}=\smallo(1), 
\qquad \frac{m_1 m_3}{m_2^2}\le \left(\frac{m_{\infty}}{m_0}\right)^4=\bigo(1),
$$
and so we can write
\begin{equation}
\label{eq:explambdamax}
\mathbb{E} [\lambda_1] = \frac{m_2}{m_1}+\frac{m_1m_3}{m_2^2}+ \smallo(1).
\end{equation}
\end{proof}


\subsection{CLT for the principal eigenvalue}
\label{sub:cltlambda}

Again consider the high probability event on which \eqref{eq:inverse} holds. Recall that from the series decomposition in \eqref{eq:masterequation} we have
\begin{align}
\lambda_1&= \frac{\la e, He\ra}{\lambda_1} + \sum_{k=0}^{L} \frac{\Exp\la e, H^ke\ra}{\lambda_1^k}
+ \sum_{k=2}^{L} \frac{ \la e, H^ke\ra- \Exp \la e, H^k e\ra}{\lambda_1^k}
+ \sum_{k>L} \frac{\la e, H^k e\ra}{\lambda_1^k}.
\label{eq:master2}
\end{align}

\begin{lemma}
The equation
\begin{equation}
\label{eq:fixedpt}
x= \sum_{k=0}^{L} \frac{\Exp\la e, H^ke\ra}{x^k}
\end{equation}
has a solution $x_0$ satisfying
\[
\lim_{n\to\infty} \frac{x_0}{m_2/m_1} = 1.
\]
\end{lemma}

\begin{proof}
Define the function $h\colon\, (0,\infty)\to\mathbb R$ by
\[
h(x)= \sum_{k=0}^{\log n} \frac{\Exp\la e, H^ke\ra}{x^k}.
\]
Since $\Exp[e^\prime He]=0$, we have
\[
h\left(\frac{xm_2}{m_1}\right)= \frac{m_2}{m_1} +\sum_{k=2}^{\log n} \frac{\Exp\la e, H^ke\ra}{(xm_2/m_1)^k}.
\]
For $x>0$,
\begin{align*}
\left|\sum_{k=2}^{\log n} \frac{\Exp[\la e, H^ke\ra]}{(xm_2/m_1)^k}\right|
&\le \sum_{k=2}^\infty \frac{1}{(xm_2/m_1)^k}\frac{m_2}{m_1} (Cm_\infty)^{k/2} \\
&= \smallo\left( \frac{m_2}{m_1} \sum_{k=2}^\infty\frac{1}{x^k (\log n)^{k\xi}}\right)
=\smallo\left( \frac{m_2}{m_1} x^{-2}\right).
\end{align*}
This shows that
\[
\lim_{n\to\infty} \frac{1}{m_2/m_1} \sum_{k=0}^{\log n} \frac{\Exp\la e, H^ke\ra}{(xm_2/m_1)^k} = 1.
\]
Hence, for any $0<\delta <1$,
\[
\lim_{n\to\infty} \frac{1}{m_2/m_1} \left[ \frac{m_2}{m_1}(1+\delta)- h\left((1+\delta) \frac{m_2}{m_1}\right)\right] = \delta.
\]
So, for large enough $n$,
\[
h\left((1+\delta) \frac{m_2}{m_1}\right) < \frac{m_2}{m_1} (1+\delta).
\]
Similarly, for any $0<\delta<1$,
\[
h\left((1-\delta) \frac{m_2}{m_1}\right) > \frac{m_2}{m_1} (1-\delta).
\]
This shows that there is a solution for \eqref{eq:fixedpt}, which lies in the interval $[ \frac{m_2}{m_1}(1-\delta), \frac{m_2}{m_1}(1-\delta)]$.
\end{proof}

\begin{lemma}
\label{lemma:crucial}
Let $x_0$ be a solution for \eqref{eq:fixedpt}. Define
\[
R_n=\lambda_1 -x_0 -\frac{\la e, H e\ra}{m_2/m_1}.
\]
Then 
$$
R_n=\smallo_\Pro\lep \frac{m_3}{m_2\sqrt{m_1}}\rip,
\qquad \Exp\left[|R_n|\right]=\smallo\lep \frac{m_3}{m_2\sqrt{m_1}}\rip.
$$
\end{lemma}

\begin{proof}[{\bf Proof of Theorem \ref{thm:CLTlambda} (II)}]
From the previous lemmas we have
\[
\lambda_1= x_0+ \frac{\la e, H e\ra}{m_2/m_1}+ R_n.
\]
Therefore
\[
\Exp[\lambda_1]= x_0+ \Exp[R_n]
\]
and
\[
\lambda_1- \Exp[\lambda_1]= \frac{\la e, H e\ra}{m_2/m_1}+ \smallo\lep\frac{m_3}{m_2\sqrt{m_1}}\rip.
\]
Hence 
\begin{equation}
\label{eq:clt-focus}
\frac{m_2}{m_1}\left(\lambda_1- \Exp[\lambda_1]\right)= \la e, H, e\ra + \smallo\lep \frac{m_3}{m_1^{3/2}}\rip.
\end{equation}

Observe that
$$
\la e, H e\ra = \sum_{i,j=1}^N h_{i,j} \frac{d_i d_j}{m_1}= 2\sum_{i\le j}h_{i,j} \frac{d_i d_j}{m_1}
$$
Let 
$$
\sigma_1^2= \sum_{i\le j} \mathrm{Var}\lep \frac{2}{m_1}h_{i,j} d_i d_j\rip=\sum_{i\le j} 
\frac{4d_i^3d_j^3}{m_1^3}\left(1- \frac{d_id_j}{m_1}\right)\sim 2 \frac{m_3^2}{m_1^3} 
\lep 1+\bigo\lep\frac{m_\infty^2}{n}\rip\rip,
$$
where we use the symmetry of the expression in the last equality. We can apply Lyapunov's central limit theorem, because $\{h_{i,j}\colon\, i\le j\}$ is an independent collection of random variables and Lyapunov's condition is satisfied, i.e., 
\begin{align*}
\lim_{n \to \infty} \frac{1}{\sigma_n^3} \sum_{i > j} \bbe \left[ \left| H(i,j) d_i d_j \right|^3\right]
\leq K\lim_{n \to \infty} \frac{m_1^{3/2}}{m_3^3}\frac{m_4^2}{m_1} =0,
\end{align*}
where $K$ is a constant that does not depend on $n$. Hence
$$ 
\frac{m_1^{3/2}\la e, H e\ra}{\sqrt{2} m_3}\overset{w}\longrightarrow N(0,1).
$$

Returning to the eigenvalue equation in \eqref{eq:clt-focus} and dividing by $\sigma_1$, we have
$$
\frac{\sqrt{m_1}{m_2}}{m_3} \left( \lambda_1- \E[ \lambda_1]\right)
= \frac{m_1^{3/2}\la e, H e\ra }{ m_3}+ \smallo(1)\overset{w}\longrightarrow N(0, 2).
$$
\end{proof}

We next prove Lemma \ref{lemma:crucial}, on which the proof of the central limit theorem relied. 

\begin{proof}
Note that by \eqref{eq:master2} and \eqref{eq:fixedpt} we can write
\begin{equation}
\label{eq:subtract1}
\lambda_1-x_0= \frac{\la e, H e\ra}{\lambda_1}+\sum_{k=2}^{L} \Exp\la e, H^k e\ra
\left( \frac{1}{\lambda_1^k}- \frac{1}{x_0^k}\right)+L_n,
\end{equation}
where
$$
L_n=\sum_{k=2}^L \frac{ \la e, H^k e\ra- \Exp \la e, H^k e\ra}{\lambda_1^k}+\sum_{k>L} 
\frac{\la e, H^k e\ra}{\lambda_1^k}.
$$
Thanks to Lemma~\ref{lem:concentrationlambda}, Lemma \ref{lemma:exppowers} and \eqref{eq:tailsum} we have
$$
L_n = \bigo\left(\frac{m_\infty(\log n)^{2\xi}}{\sqrt{n}m_2/m_1}\right).
$$ 
Note that $L_n= \smallo( \frac{m_3}{m_2 \sqrt{m_1}})$. Indeed, using $m_3\ge n m_0^3$ and Assumption \ref{ass:deg}(D1), we get
\[
\frac{m_\infty(\log n)^{2\xi} m_2\sqrt{m_1}}{\sqrt{n}(m_2/m_1) m_3}
\le \frac{m_{\infty}^{5/2} n^{3/2} (\log n)^{2\xi}}{\sqrt{n} m_0  n m_0^3 (\log n)^\xi}=\frac{ m_{\infty}^{5/2}(\log n)^{\xi}}{m_0^4 }
=\bigo\lep \frac{(\log n)^{\xi}}{m_0^{3/2}}\rip.
\]

Observe that \eqref{eq:subtract1} can be rearranged as 
\[
(\lambda_1-x_0)= \frac{\la e, He\ra}{\lambda_1}-\sum_{k=2}^{L} (\lambda_1-x_0)
\Exp\la e, H^ke\ra\lambda_1^{-k} x_0^{-k} \sum_{j=0}^{k-1} x_0^{k-1-j}+ L_n.
\]
Hence, bringing the second term from the right to the left, we have
\[
(\lambda_1-x_0)\left[ 1+ \sum_{k=2}^{L} \Exp\la e, H^ke\ra\lambda_1^{-k} x_0^{-k}
\sum_{j=0}^{k-1} x_0^{k-1-j}\right]= \frac{\la e, H e\ra}{\lambda_1} + L_n.
\]
Using the bounds on $\lambda_1$ and $x_0$, we get
\begin{align*}
&\left|\sum_{k=2}^{L} \Exp\la e, H^ke\ra\lambda_1^{-k} x_0^{-k} \sum_{j=0}^{k-1} x_0^{k-1-j}\right|
\le \sum_{k=2}^{L} \frac{k}{(m_2/m_1)^{k+1}} \Exp\la e, H^ke\ra\\
&\le \sum_{k=2}^{L} \frac{k}{(m_2/m_1)^{k+1}}\frac{m_2}{m_1} (Cm_\infty)^{k/2}
= \bigo\lep\frac{m_{\infty}}{(m_2/m_1)^2(\log n)^{2\xi-1}}\rip=\smallo(1).
\end{align*}
We can therefore write 
\[
\lambda_1-x_0= \frac{\la e,H e\ra}{\lambda_1}+L_n,
\]
where $L_n= \smallo_P( \frac{m_3}{m_2 \sqrt{m_1}})$. Finally, to go to $R_n$, note that
\begin{equation}
R_n= \lambda_1-x_0- \frac{\la e, H e\ra} {m_2/m_1} =\la e, H e\ra \left( \frac{1}{\lambda_1}- \frac{1}{m_2/m_1}\right)+ L_n.
\end{equation}
To bound $R_n$, it is enough to show that the first term on the right-hand side is whp bounded by $\frac{m_3}{m_2 \sqrt{m_1}}$. Using Lemma~\ref{lemma:exppowers} (for $k=1$) and \eqref{eq:normH2}, we have whp
\begin{equation}
\label{eq:firstbound}
\frac{\left| \la e, H e\ra \right| |\lambda_1- m_2/m_1|}{\lambda_1 m_2/m_1}
\le \frac{\sqrt{m_\infty}(\log n)^{\xi}}{\sqrt{n}}\frac{\sqrt{m_{\infty}}}{(m_2/m_1)}.
\end{equation}
Using again Assumption \ref{ass:deg}(D1), $m_3\ge n m_0^3$, $m_1\le n m_\infty$ and $m_2\le n m_{\infty}^2$, we get that
\[
\frac{m_\infty (\log n)^{\xi}}{\sqrt{n}(m_2/m_1)}\frac{m_2\sqrt{m_1}}{m_3}\le \left(\frac{m_\infty}{m_0}\right)^3
\frac{c}{\sqrt{m_{\infty}}}=\smallo(1).
\]
This controls the right-hand side of \eqref{eq:firstbound}, and hence $R_n=\smallo( \frac{m_3}{m_2\sqrt{m_1}})$ whp.

We want to show that the latter is negligible both pointwise and in expectation. We already have that this is so whp on $R_n$. We want to show that the same bound holds in expectation. Let $\mathcal{A}$ be the high probability event of Lemma \ref{lem:concentrationlambda} and \ref{lemma:exppowers}, and write
\[
\Exp[|R_n|] =  \Exp[|R_n| \textbf{1}_{\mathcal{A}^c}] + \Exp[|R_n|\textbf{1}_{\mathcal{A}}],
\]
where $\textbf{1}_{\mathcal{A}}$ is the indicator function of the event $\mathcal{A}$. Since all the bounds hold on the high probability event $\mathcal A$, it is immediate that
\[
\Exp[|R_n|\textbf{1}_{\mathcal{A}}]
=\smallo\lep\frac{m_3}{\sqrt{m_1} m_2}\rip.
\]
The remainder can be bounded via the Cauchy-Schwarz inequality, namely,
\[
\Exp[|R_n| \textbf{1}_{\mathcal{A}^c}] \leq \left(\Exp[|R_n|^2]\Exp[\textbf{1}_{\mathcal{A}^c}]\right)^{\frac{1}{2}}\leq\lep\Exp\left[|R_n|^2\right]e^{-\nu(\log n)^\xi}\rip^{\frac{1}{2}}.
\]
We see that if $\Exp[|R_n|^2]=\smallo(\ee^{-\nu(\log n)^\xi})$, then we are done. Expanding, we see that
\[
\Exp[|R_n|^2] = \Exp\left[\left|\lambda_1 - x_0 - \frac{\la e, H e\ra}{m_2/m_1} \right|^2\right] \le n^C
\]
for some $C>0$, where we use that 
\begin{align*}
\Exp[(\lambda_1^2)] \leq \Exp[\Tr A^2] = \sum_{i,j=1}^N \Exp[(A(i,j))^2] \leq m_\infty n
\end{align*} 
and the trivial bound $|\la e, He\ra|\leq n^{C_*}$ for some $C_*<C$. Hence we have $\left(\Exp[|R_n|^2]\Exp[\textbf{1}_{A^c}]\right)^{\frac{1}{2}} \le \ee^{-\nu(\log n)^{\xi}}$ and 
\[
\Exp[|R_n|] =\smallo\lep\frac{m_3}{\sqrt{m_1} m_2}\rip.
\]
\end{proof}


\section{Proof of Theorem~\ref{thm:CLTv}}
\label{sub:cltvect}

In this section we study the properties of the principal eigenvector. Let $v_1$ be the normalized principal eigenvector, i.e., $\|v_1\|=1$, and let $e$ be as defined in \eqref{eq:def-e}. Recall from \eqref{eq:eigenvector} that
\[
\lambda_1 \left(1-\frac{H}{\lambda_1}\right) v_1= e\langle e,v_1\rangle,
\]
and after inversion (which is possible on the high probability event) we have
\[
v_1=\frac{\langle e,v_1\rangle}{\lambda_1}(1-H/\lambda_1)^{-1} e.
\]
If $K$ denotes the normalization factor, then we can rewrite the above equation whp as the series
\begin{equation}
\label{eq:vector_expansion}
v_1=\frac{K}{\lambda_1}\sum_{k=0}^\infty \frac{H^ke}{\lambda_1^k}.
\end{equation}
Our first step is to determine the value of $K$ in \eqref{eq:vector_expansion}. We adapt the results from \cite{erdosSpectralStatisticsErdos2013} to derive a component-wise central limit theorem in the inhomogeneous setting described by \eqref{eq:chunglu} under Assumption \ref{ass:deg}. By the normalization of $v$,
\begin{equation}
\label{eq:K}
1 = \langle v_1 , v_1\rangle= \frac{K^2}{\lambda_1^2}\left \langle \sum_{k=0}^\infty
\frac{H^k}{\lambda_1^k} e,\sum_{\ell=0}^\infty \frac{H^{\ell}}{\lambda_1^{\ell}} e \right\rangle
= \frac{K^2}{\lambda_1^2}\sum_{k=0}^\infty \frac{(k+1)\left\langle e,{H^k}e\right\rangle}{\lambda_1^k},
\end{equation}
where we use the symmetry of $H$. 

The following lemma settles Theorem \ref{thm:CLTv}(I).

\begin{lemma}
\label{lemma:parallel}
Under Assumption \ref{ass:deg}, and with $\tilde{e}=e\sqrt{\frac{m_1}{m_2}}$, whp
\begin{equation}
\langle \tilde{e},v_1\rangle=1+\smallo(1).
\end{equation}

\end{lemma}

\begin{proof}
Recall that $L=\lfloor \log n\rfloor$. We rewrite \eqref{eq:K} as
\begin{equation}
\label{eq:K_2}
\begin{split}
\lep\frac{\lambda_1}{K}\rip^2
&= \sum^{L}_{k= 0}\frac{(k+1)}{\lambda_1^k}\,\Exp\left[\left\langle e, H^k e\right\rangle\right]
+\sum_{k=1}^{L} \frac{(k+1)}{\lambda_1^k}\left|\left\langle e, H^k e\right\rangle
-\Exp\left[\left\langle e,H^k e\right\rangle\right]\right| \\
&\qquad \qquad +\sum_{k=L+1}^\infty \frac{(k+1)}{\lambda_1^k}\left\langle e, H^k e\right\rangle.
\end{split}
\end{equation}
We first show that the last two parts are negligible and then show that the main term of the first part is the term with $k=0$, i.e., $\langle e,e\rangle=m_2/m_1$. 

The last term in \eqref{eq:K_2} is dealt with as follows. Using \eqref{eq:normbound}, we have whp
\[
\begin{split}
\sum_{k=L+1}^\infty \frac{(k+1)}{\lambda^k}\left\langle e,H^ke\right\rangle
&\leq\sum_{k=L+1}^\infty (k+1) \frac{\|e\|^2\|H\|^k}{(m_2/m_1)^k}
\leq \frac{m_2}{m_1}\sum_{k=L+1}^\infty (k+1) (1-C_0)^k\\
&\leq\frac{m_2}{m_1} (\log n+2) \ee^{-c'\log n}\frac{1}{C_0^2}
\end{split}
\]
with $c'=-\log(1-C_0)$, where we use that $\sum_{k=0}^\infty (k+1)(1-c)^k=1/c^2$ for $|1-c|<1$.

We tackle the second sum in \eqref{eq:K_2} by using Lemma \ref{lemma:exppowers}. Indeed, whp we have 
\[
\begin{split}
\sum_{k=1}^L \frac{(k+1)}{\lambda_1^k}\left|\left\langle e, H^k e\right\rangle
-\Exp\left[\left\langle e,H^ke\right\rangle\right]\right|
&\leq\sum_{k=1}^L (k+1) \frac{Cm_\infty^{k/2}(\log n)^{k\xi}}{\sqrt{n}}\lep\frac{m_2}{m_1}\rip^{1-k}\\
&\leq\frac{C'\sqrt{m_\infty}(\log n)^\xi(\log n +1)}{\sqrt{n}}\leq \frac{C'\sqrt{m_\infty}(\log n)^{2\xi}}{\sqrt{n}},
\end{split}
\]
where the constant varies in each step. By Assumption \ref{ass:deg}(D1), the last term goes to zero.

As to the first term, note that by \eqref{lemma:expectationbounds} for $k\geq3$ we have
\begin{align*}
\sum^{L}_{k=3}\frac{(k+1)}{\lambda_1^k}\Exp\left[\left\langle e, H^k e\right\rangle
\right]&\leq \sum_{k=3}^{L} (k+1) 
\lep\frac{m_2}{m_1}\rip^{-k+1} (Cm_\infty)^{k/2} \\
&\le  \sum_{k=3}^L \frac{C m_{\infty}^{k/2}}{(m_2/m_1)^{(k-1)}}=\bigo\lep\frac{1}{\sqrt{m_\infty}}\rip.
\end{align*}
The term with $k=1$ is zero, while for $k=2$ we have
\[
3\frac{\Exp \langle e,H^2e\rangle}{\lambda^2_1}\leq c \frac{m_1m_3}{m_2^2}= \bigo\lep1\rip
\]
for some constant $c$. After substituting these results into \eqref{eq:K_2}, we find
\begin{equation}\label{K-2bound}
\lep\frac{\lambda_1}{K}\rip^2= \frac{m_2}{m_1}\lep1+\bigo\lep\frac{1}{m_2/m_1}\rip\rip
\end{equation}
and the proof follows by normalizing the vector $e$ and using \eqref{eq:vector_expansion}.
\end{proof}

The following lemma is an immediate consequence of \eqref{eq:vector_expansion} and Lemma \ref{lemma:parallel}.

\begin{lemma}
\label{v_expansion}
Under Assumptions \ref{ass:deg}, whp
\begin{equation}
\label{eq:v_expansion}
v_1 = \lep1+\bigo\lep\frac{m_1}{m_2}\rip\rip \sqrt{\frac{m_1}{m_2}}
\sum_{k=0}^\infty\frac{H^k}{\lambda_1^k}e.
\end{equation}
\end{lemma}

In order to estimate how the components of $v_1$ concentrate, we need the following lemma.

\begin{lemma}
\label{lem:conccomponent}
For $1 \leq k \leq L$, whp
\[
|H^k e(i)| =\left| \frac{1}{\sqrt{m_1}}
\sum_{i_1,\dots,i_k}h_{i i_1}h_{i_1i_2}\dots h_{i_{k-1}i_k}d_{i_k}\right| 
\leq\frac{m_\infty}{\sqrt{m_1}}\lep(\log n)^{\xi} \sqrt{m_\infty} \rip^k.
\]
\end{lemma}

\noindent 
The proof of this lemma is a direct consequence of Lemma \ref{lemma:exppowers}, is similar to \cite[Lemma 7.10]{erdosSpectralStatisticsErdos2013} and therefore we skip it. An immediate corollary of the above estimate is the delocalized behaviour of the largest eigenvector stated in Theorem \ref{thm:CLTv}(II).

\begin{lemma}
Let $v_1$ be the normalized principal eigenvector, and $\tilde{e}=e\sqrt{\frac{m_1}{m_2}}$. Then whp 
$$
\|v_1- \tilde{e}\|_{\infty} \leq \bigo\left(\frac{(\log n)^{\xi}}{\sqrt{n m_{\infty}}}\right).
$$
\end{lemma}

\begin{proof}
Recall from \eqref{eq:K_2} that
\[
v_1(i) =\frac{K}{\lambda_1} \sum_{k=0}^\infty \frac{H^ke(i)}{\lambda_1^k}
= \frac{K}{\lambda_1} e(i) + \frac{K}{\lambda_1} \sum_{k=1}^L \frac{H^ke(i)}{\lambda_1^k}
+\frac{K}{\lambda_1} \sum_{k=L+1}^\infty \frac{H^ke(i)}{\lambda_1^k}.
\]
The last term is negligible whp, because it is the tail sum of a geometrically decreasing sequence. For the sum over $1 \leq k \leq L$ fwe can use Lemma \ref{lem:conccomponent} and the fact that $K/\lambda_1= \sqrt{\frac{m_1}{m_2}}+ \smallo(1)$ whp. So we have
$$
\frac{K}{\lambda_1} \sum_{k=1}^L \frac{H^ke(i)}{\lambda_1^k}\le \frac{m_{\infty}}{\sqrt{n}m_0} \frac{(\log n)^{\xi}}{\sqrt{m_{\infty}}}\le \bigo\left(\frac{(\log n)^{\xi}}{\sqrt{nm_{\infty}}}\right).
$$
The first term whp is
$$
\frac{K}{\lambda_1} e(i) = \tilde{e}(i)+ \smallo(1)
$$
and the error is uniform over all $i$. Indeed, whp
\begin{equation}
\label{firsttermK}
\left| \frac{K}{\lambda_1} e(i) - \frac{K}{m_2/m_1} e(i)\right| 
\le \frac{K d_i}{\sqrt{m_1}} \frac{\left|\lambda_1- m_2/m_1\right|}{(m_2/m_1)^2}
\le \sqrt{\frac{m_2}{m_1}}\frac{cm^{3/2}_\infty}{\sqrt{m_0n}} \frac{c'}{m_{\infty}^2} =\bigo\lep\frac{1}{\sqrt{n m_{\infty}}}\rip,
\end{equation}
where we use Assumption \ref{ass:deg}, Remark \ref{rmk:m2/m1} and \eqref{eq:normH2}. Since the detailed computations are similar to the previous arguments, we skip the details. 
\end{proof}

We next prove the central limit theorem for the components of the eigenvector stated in Theorem \ref{thm:CLTv}(IV). 

\begin{theorem}
Under Assumption \ref{ass:deg}, with the extra assumtion $m_{\infty}\gg (\log n)^{4\xi}$,
$$ 
\sqrt{\frac{m_2^3}{d_i m_3 m_1}} \Big(v_1(i)-\frac{d_i}{\sqrt{m_2}}\Big) \overset{w}\rightarrow \mathcal N(0,1).
$$
\end{theorem}

\begin{proof}
First we compute $\Exp[v_1(i)]$, and afterwards we show that the CLT holds componentwise. 

We use the law of total expectation. Conditioning on the high probability event $\mathcal E$ in Lemma \ref{lem:concentrationlambda}, we can write the expectation of the normalized eigenvector $v_1$ as
\[
\Exp[v_1(i)]=\Exp[v_1(i)|\mathcal{E}]\,\Pro(\mathcal{E})+\Exp[v_1(i)|\mathcal{E}^c]\,\Pro(\mathcal{E}^c).
\]
Because the components of a normalized $n$-dimensional vector are bounded, we know that
\[
\Exp[v_1(i)]=\Exp[v_1(i)|\mathcal{E}]\,\Pro(\mathcal{E})+\bigo\lep e^{-c_\nu (\log n)^\xi}\rip
\]
for some suitable constant $c_\nu>0$, dependent on $\nu$ and on the the bound on $v_1(i)$. On $\mathcal{E}$, we can expand $v_1$ as
\begin{align*}
v_1(i) = \frac{K}{\lambda_1} \lep e(i)
+ \frac{(He)(i)}{\lambda_1}+\frac{(H^2e)(i)}{\lambda_1^2}
+ \sum_{k=3}^\infty \frac{(H^ke)(i)}{\lambda_1^k} \rip.
\end{align*}
Using the notation $\Exp_{\mathcal E}$ for the conditional expectation on the event $\mathcal E$, we have
\begin{align*}
&\Exp_\mathcal{E}[v_1(i)]=\Exp_\mathcal{E}\left[\frac{K}{\lambda_1}  e(i)\right]
+ \Exp_\mathcal{E}\left[\frac{K}{\lambda_1} \frac{(He)(i)}{\lambda_1}\right]+
+ \Exp_\mathcal{E}\left[\frac{K}{\lambda_1} \sum_{k=2}^\infty \frac{(H^ke)(i)}{\lambda_1^k} \right].
\end{align*}

For the first term we have, using \eqref{K-2bound},
\begin{align*}
\Exp_{\mathcal{E}}\left[ \frac{K}{\lambda_1}e_i\right]
&= \Exp_{\mathcal{E}}\left[ \frac{1}{\sqrt{m_2/m_1}}e_i\right]+\bigo\lep\frac{d_i}{\sqrt{m_1} (m_2/m_1)^{3/2}}\rip
=\frac{d_i}{\sqrt{m_2}} +\bigo\lep \frac{d_i}{\sqrt{m_1} (m_2/m_1)^{3/2}}\rip.
\end{align*}
For the term corresponding to $k=1$, we know that $\Exp[(He)(i)]=0$ by construction on the whole space. However, under the event $\mathcal{E}$ we can show that its contribution is exponentially negligible. We have 
\begin{align*}
\Exp_\mathcal{E}\left[\frac{K}{\lambda_1}\frac{(He)(i)}{\lambda_1}\right]
=\Exp_\mathcal{E}\left[\frac{K}{\lambda_1}\frac{\sum_jh_{ij}d_j}{\sqrt{m_1}\lambda_1}\right]
=\Exp_\mathcal{E}\left[\frac{\lep1+\bigo\lep \frac{1}{m_2/m_1}\rip\rip}{\sqrt{m_2/m_1}}
\lep\frac{\sum_j h_{ij}d_j}{\sqrt{m_1}(m_2/m_1)}\right.\right.\\\left.\left.
+\frac{\sum_j h_{ij}d_j}{\sqrt{m_1}}\frac{|\lambda_1-(m_2/m_1)|}{(m_2/m_1)^2}\rip\right].
\end{align*}
Since $m_2/m_1\to\infty$, there exists a constant $\tilde C$ such that
\[
\frac{\lep1+\bigo\lep1/(m_2/m_1)\rip\rip}{\sqrt{m_2/m_1}}\leq \tilde{C} \frac{1}{\sqrt{m_2/m_1}}.
\]
We can therefore write
\begin{align*}
\Exp_\mathcal{E}\left[\frac{K}{\lambda_1}\frac{\sum_jh_{ij}d_j}{\sqrt{m_1}\lambda_1}\right]
&\leq \tilde{C}\frac{1}{\sqrt{m_2/m_1}}\Exp_\mathcal{E}\left[\frac{\sum_j h_{ij}d_j}{\sqrt{m_1}(m_2/m_1)}
+\frac{\sum_j h_{ij}d_j}{\sqrt{m_1}}\frac{|\lambda_1-(m_2/m_1)|}{(m_2/m_1)^2}\right]\\
&\leq\Exp_\mathcal{E}\left[\frac{\sum_j h_{ij}d_j}{\sqrt{m_1}(m_2/m_1)}\right]
+\Exp_\mathcal{E}\left[\frac{\sum_j h_{ij}d_j}{\sqrt{m_1}}\frac{|\lambda_1-(m_2/m_1)|}{(m_2/m_1)^2}\right]\\
&\leq \Exp_\mathcal{E}\left[\sum_j h_{ij}d_j\right]\lep\frac{1}{\sqrt{m_1}(m_2/m_1)}
+\frac{\sqrt{m_\infty}}{\sqrt{m_1}(m_2/m_1)}\rip.
\end{align*}
Here we use \eqref{eq:normH2} to bound the difference $|\lambda_1-(m_2/m_1)|$. Next, write 
\begin{align*}
0
&=\Exp\left[\sum_j h_{ij}d_j\right]=\Exp_\mathcal{E}\left[\sum_j h_{ij}d_j\right]\Pro(\mathcal{E})
+\Exp_{\mathcal{E}^c}\left[\sum_j h_{ij}d_j\right]\Pro(\mathcal{E}^c)\\
&\leq\Exp_\mathcal{E}\left[\sum_j h_{ij}d_j\right]\Pro(\mathcal{E})+m_1\Pro(\mathcal{E}^c)
=\Exp_\mathcal{E}\left[\sum_j h_{ij}d_j\right]\Pro(\mathcal{E})+\bigo\lep e^{-c_\nu (\log n)^\xi}\rip,
\end{align*}
where $c_\nu$ is a constant depending on $\nu$, and we use that $|h_{ij}|\leq 1$ and $m_1=\bigo\lep e^{3/2\log n}\rip$. We can therefore conclude that
\[
\Exp_\mathcal{E}\left[\frac{(He)(i)}{\lambda_1}\right]=\bigo\lep e^{-c^\prime_\nu (\log n)^\xi}\rip,
\]
where $c^\prime_\nu>0$ is a suitable constant depending on $\nu$, and possibly different from $c_\nu$.

To bound the remaining expectation terms, we use Lemma \ref{lem:conccomponent}, which gives a bound on $(H^ke)(i)$ on the event $\mathcal{E}$. As before, we break up the sum into two contributions:
\begin{align*}
\Exp_\mathcal{E}\left[\frac{K}{\lambda_1} \sum_{k=2}^\infty \frac{(H^ke)(i)}{\lambda_1^k}\right]
=\Exp_\mathcal{E}\left[\frac{K}{\lambda_1} \sum_{k=2}^L \frac{(H^ke)(i)}{\lambda_1^k}\right]
+\Exp_\mathcal{E}\left[\frac{K}{\lambda_1} \sum_{k=L}^\infty \frac{(H^ke)(i)}{\lambda_1^k}\right].
\end{align*}
For the second term we have 
\begin{equation}
\label{eq:vectormoreL}
\begin{aligned}
\sum_{k=L+1}^\infty \frac{\lep H^ke\rip(i)}{\lambda_1^k}
\leq C  \sqrt{\frac{m_2}{m_1}}\,\ee^{-C_c\log n},
\end{aligned}
\end{equation}
where we use \eqref{eq:normbound} and $C_c=|\log (1-C_0)|$. The first term can be bounded via Lemma \ref{lem:conccomponent}, which gives
\begin{equation}
\label{eq:vectorlessL}
\begin{aligned}
\sum_{k=2}^L \frac{(H^ke)(i)}{\lambda_1^k} 
&\leq\sum_{k=2}^L 
\frac{m_\infty \lep(\log n)^\xi \sqrt{m_\infty}\rip^k}{\sqrt{m_1}(m_2/m_1)^k}
=\bigo\lep \frac{(\log n)^{2\xi}}{\sqrt{m_1}}\rip.
\end{aligned}
\end{equation}
Using the above bounds, taking expectations and using \eqref{K-2bound}, we get
\[
\Exp_\mathcal{E}\left[\frac{K}{\lambda_1} \sum_{k=2}^\infty \frac{(H^ke)(i)}{\lambda_1^k}\right]
=\bigo\lep\frac{(\log n)^{2\xi}}{\sqrt{m_2}}\rip.
\]
Thus, we have obtained that
\[
\Exp[v_1(i)]=\frac{d_i}{\sqrt{m_2}}+\bigo\lep\frac{(\log n)^{2\xi}}{\sqrt{m_2}}\rip,
\]
which settles Theorem \ref{thm:CLTv}(III).

We can write
\[
v_1(i)-\frac{d_i}{\sqrt{m_2}}= \frac{\lep 1+\bigo\lep\frac{1}{m_2/m_1}\rip\rip e(i)}{\sqrt{m_2/m_1}} 
-\frac{d_i}{\sqrt{m_2}}+ \frac{K}{\lambda_1} \frac{(He)(i)}{\lambda_1} +\bigo\lep\frac{(\log n)^{2\xi}}{\sqrt{m_2}}\rip,
\]
where we replace the last terms of the expansion of $v_1$ by the bounds derived above (note that these bounds are of the same order as the ones obtained for the same terms in expectation). The first term of the centered quantity $v_1(i)-d_i/\sqrt{m_2}$ is given by  
\[
\frac{\lep 1+\bigo\lep\frac{1}{m_2/m_1}\rip\rip e(i)}{\sqrt{m_2/m_1}}
=\bigo\lep\frac{d_i}{\sqrt{m_1} (m_2/m_1)^{3/2}}\rip.
\] 
This last error can be easily seen to be $\smallo\lep\frac{(\log n)^{2\xi}}{\sqrt{m_2}}\rip$. We can therefore write
\[
v_1(i)-\Exp[v_1(i)]= \frac{K}{\lambda_1} \frac{(He)(i)}{\lambda_1} +\bigo\lep\frac{(\log n)^{2\xi}}{\sqrt{m_2}}\rip.
\]

We proceed to show that the first term on the right-hand side of the above equality gives a CLT when the expression is rescaled by an appropriate quantity, and the error term goes to zero. It turns out that 
\[
s_n^2(i)=\mathrm{Var}\lep\sum_j h_{ij} d_j\rip = \sum_j \frac{d_id^3_j}{m_1}\lep1+\bigo\lep\frac{1}{m_0}\rip\rip
\sim \frac{d_im_3}{m_1}.
\]
Multiplying by $\sqrt{\frac{m_2^3}{d_im_3m_1}}$, we have
\[
\sqrt{\frac{m_2^3}{d_im_3m_1}} \Big(v_1(i)-\langle\tilde{e},v_1\rangle\tilde{e}(i)\Big)
= \frac{1}{s_n}\sum_j h_{ij}d_j +\bigo\lep\sqrt{\frac{m_2^2 (\log n)^{4\xi}}{d_im_3m_1}}\rip.
\]
The error term is 
\[
\sqrt{\frac{m_2^2 (\log n)^{4\xi}}{d_im_3m_1}}=\bigo\lep\frac{(\log n)^{2\xi}}{\sqrt{m_0}}\rip =\smallo(1),
\]
where last inequality follows from the assumption that $m_0\gg(\log n)^{4\xi}$. We now apply Lindeberg's CLT to the term $\frac{\sum_j h_{ij}d_j}{s_n}$. The Lindeberg condition for the CLT reads 
\begin{equation}
\label{eq:Lindeberg}
\lim_{n\to\infty}\frac{1}{s_n^2(i)}\sum_{j}^n\Exp\left[(h_{i j}d_j)^2 \,
\one_{\{|h_{i j}d_j|\geq \epsilon s_n(i) \}}\right]=0.
\end{equation}
Defining $\sigma^2_j(i)=\Var(h_{ij}d_j)$, we note that
\[
\lim_{n\to\infty}\frac{\sigma^2_j(i)}{s^2_n(i)}
=\lim_{n\to\infty}\frac{d_i d_j^3 m_1}{m_1 m_3 d_i}\leq \lim_{n\to\infty}\frac{m_\infty^3}{m_3}
\leq\lim_{n\to\infty}\frac{m_\infty^3}{nm_0^3}=0.
\]
 
Let us finally examine the event
\[
|h_{ij}d_j|\geq \epsilon s_n(i)=\epsilon \sqrt{\frac{d_im_3}{m_1}}\iff|h_{ij}|\geq \epsilon \sqrt{\frac{m_3}{m_1}\frac{d_i}{d_j^2}}.
\]
By definition, $|h_{ij}|<1$. If we show that 
\[
\lim_{n\to\infty} \sqrt{\frac{m_3}{m_1}\frac{d_i}{d_j^2}}=\infty,
\]
then for all $\epsilon>0$ there exists $n_\epsilon$ such that the event
\[
\epsilon\sqrt{\frac{m_3}{m_1}\frac{d_i}{d_j^2}}>1>|h_{ij}|
\]
has probability $1$. Indeed,
\[
\lim_{n\to\infty}\epsilon \sqrt{\frac{m_3}{m_1}\frac{d_i}{d_j^2}}
> \lim_{n\to\infty}\epsilon\sqrt{\frac{nm_0^4}{nm_\infty^3}}
\geq \lim_{n\to\infty}\epsilon\, C\sqrt{m_0} = \infty
\]
for a suitable constant $C$. Thus, \eqref{eq:Lindeberg} holds.
\end{proof}



\end{document}